\documentclass[10pt]{article}

\usepackage{amsxtra,amssymb,amsthm,amsmath,latexsym}

\usepackage{graphicx}
\newtheorem{thm}{Theorem}[section]

\newtheorem{lem}[]{Lemma}

 \newcommand{\thmref}[1]{Theorem~\ref{#1}}
 \newcommand{\lemref}[1]{Lemma~\ref{#1}}

\newcommand{\R}{{\mathbb R}}
\newcommand{\C}{{\mathbb C}}

\newcommand{\bee}{\begin{equation*}}
\newcommand{\eee}{\end{equation*}}
\newcommand{\be}{\begin{equation}}
\newcommand{\ee}{\end{equation}}
\newcommand{\pn}{\par\noindent}
\title{Dynamical 
Systems Method (DSM) for solving nonlinear operator 
equations in Banach spaces}
\author{A G Ramm\\
\small Department of Mathematics\\[-0.8ex]
\small Kansas State University, Manhattan, KS 66506-2602, USA\\[-0.8ex]
\small \texttt{ramm@math.ksu.edu}\\
}

\begin{document}
 \date{} \maketitle \begin{abstract} Let $F(u)=h$ be an operator equation
in a Banach space $X$, $\|F'(u)-F'(v)\|\leq \omega(\|u-v\|)$, where
$\omega\in C([0,\infty))$, $\omega(0)=0$, $\omega(r)>0$ if $r>0$,
$\omega(r)$ is strictly growing on $[0,\infty)$. Denote $A(u):=F'(u)$,
where $F'(u)$ is the Fr\'{e}chet derivative of $F$, and $A_a:=A+aI.$
Assume that (*) $\|A^{-1}_a(u)\|\leq \frac{c_1}{|a|^b}$, $|a|>0$, $b>0$,
$a\in L$. Here $a$ may be a complex number, and $L$ is a smooth path on
the complex $a$-plane, joining the origin and some point on the complex
$a-$plane, $0<|a|<\epsilon_0$,  
where
$\epsilon_0>0$ is a small fixed number,
 such that for any $a\in L$ estimate (*) holds. 
It is proved that the DSM (Dynamical Systems Method) \bee
\dot{u}(t)=-A^{-1}_{a(t)}(u(t))[F(u(t))+a(t)u(t)-f],\quad u(0)=u_0,\
\dot{u}=\frac{d u}{dt}, \eee converges to $y$ as $t\to +\infty$, where
$a(t)\in L,$ $F(y)=f$, $r(t):=|a(t)|$, and $r(t)=c_4(t+c_2)^{-c_3}$, where
$c_j>0$ are some suitably chosen constants, $j=2,3,4.$ Existence of a
solution $y$ to the equation $F(u)=f$ is assumed. It is also assumed that
the equation $F(w_a)+aw_a-f=0$ is uniquely solvable for any $f\in X$,
$a\in L$, and $\lim_{|a|\to 0,a\in L}\|w_a-y\|=0.$
\end{abstract} 
\pn{\\ MSC 2000, 47J05, 47J06, 47J35  }

\section{Introduction} Consider an operator equation \be\label{e1} F(u)=f,
\ee where $F$ is an operator in a Banach space $X$. By $X^*$ denote the
dual space of bounded linear functionals on $X$.

Assume that $F$ is continuously Fr\'{e}chet differentiable, $F'(u):=A(u)$,
and 
\be\label{e2} \|A(u)-A(v)\|\leq \omega(\|u-v\|),\quad
\omega(r)=c_0r^\kappa,\quad \kappa\in (0,1], \ee $c_0>0$ is a constant.
The function $\omega(r)$, in general, is a continuous strictly growing
function, $\omega(0)=0$.

Assume that \be\label{e3} \|A_a^{-1}(u)\|\leq \frac{c_1}{|a|^b};\ |a|>0,\
A_a:=A+aI,\ c_1=const>0,\ b>0. \ee Here $a$ may be a complex number,
$|a|>0$, and there exists a smooth path $L$ on the
complex plane $\C$, such that for any $a\in L$, $|a|<\epsilon_0$, where
$\epsilon_0>0$ is a small fixed number independent of $u$, estimate 
\eqref{e3} holds, and $L$ 
joins the origin and some point $a_0$, $0<|a_0|<\epsilon_0$. 
Assumption \eqref{e3} holds if there is a smooth path $L$ on a complex 
$a$-plane, consisting of regular points of the operator $A(u)$, such
that the norm of the resolvent $A_a^{-1}(u)$ grows, as $a\to 0$,
not faster than a power $|a|^{-b}$. Thus, assumption \eqref{e3}
is a weak assumption. For example, assumption \eqref{e3} is satisfied for
the class of linear operators $A$, satisfying the spectral assumption, 
introduced in \cite{R499}, Chapter 8. This spectral assumption 
says, that the set $\{a: |\arg a-\pi|\leq \phi_0, \,\, 0<|a|<\epsilon_0\}$
consists of the regular points of the operator $A$. This assumption 
implies the estimate $||A_a^{-1}||\leq \frac{c_1}{a},\,\, 0<a<\epsilon_0,$
similar to estimate \eqref{e3}. 
 
Assume additionally that
the equation 
\be\label{e4} F(w_a)+aw_a-f=0,\quad a\in L, \ee 
is uniquely
solvable for any $f\in X$, and \be\label{e5} \lim_{a\to 0,a\in
L}\|w_a-y\|=0,\quad F(y)=f. \ee 
{\it All the above assumptions are standing and
are not repeated in the formulation of \thmref{thm1}, which is our main
result.} 

These assumptions are satisfied, e.g., if $F$ is a monotone
operator in a Hilbert space $H$ and $L$ is a segment $[0,\epsilon_0]$,
in which case $c_1=1$ and $b=1$ (see \cite{R499}).

Every equation \eqref{e1} with a linear, closed, densely defined in 
a Hilbert space $H$
operator $F=A$ can be reduced to an equation with a monotone operator
$A^*A$, where $A^*$ is the adjoint to $A$. The operator $T:=A^*A$ is
selfadjoint and densely defined in $H$.  If $f\in D(A^*)$, where $D(A^*)$ 
is the domain of $A^*$, then the equation $Au=f$ is equivalent to 
$Tu=A^*f$, provided that $Au=f$ has a solution, i.e., $f\in R(A)$, where 
$R(A)$ is the range of $A$.
Recall that $D(A^*)$ is dense in $H$ if $A$ is closed and densely
defined in $H$. 
If $f\in R(A)$ but $f\not\in D(A^*)$, then equation $Tu=A^*f$ still makes
sense and its normal solution $y$, i.e., the solution with minimal
norm, can be defined as 
\be\label{e6} y=\lim_{a\to 0}T_a^{-1}A^*f.
\ee 
One proves that $Ay=f$, and $ y\perp N(A),$
where $N(A)$ is the null-space of $A$. These
results are proved in \cite{R500}, \cite{R504}, \cite{R522}.

Our aim is to prove convergence of the DSM (Dynamical Systems Method) for
solving equation \eqref{e1}, which is of the form: 
\be\label{e7}
\dot{u}=-A_{a(t)}^{-1}[F(u(t))+a(t)u(t)-f],\quad u(0)=u_0, \ee where
$u_0\in X$ is an initial element, $a(t)\in C^1[0,\infty)$, $a(t)\in L$.
Our main result is formulated in \thmref{thm1}, in Section 2. 

The DSM for solving operator equations has been developed in
the monograph \cite{R499} and in a series of papers 
\cite{R500}-\cite{R545}. 
It was used as an efficient computational tool in
\cite{R539}-\cite{R574}.
One of the earliest papers on the continuous analog of
Newton's method for solving well-posed nonlinear operator equations was 
\cite{G}.

The novel points in the current paper include the larger class of 
the operator 
equations than earlier considered, and the weakened assumptions on the 
smoothness of the nonlinear operator $F$. While in \cite{R499}
it was often assumed that $F^{\prime \prime}(u)$ is locally bounded, in 
the current paper
a weaker assumption \eqref{e2} is used.    

Our proof of \thmref{thm1}
uses the following result from \cite{R558}. \begin{lem}\label{lem1} Assume
that $g(t)\geq 0$ is continuously differentiable on any interval $[0,T)$
on which it is defined and satisfies the following inequality:
\be\label{e8} \dot{g}(t)\leq -\gamma(t)g(t)+\alpha(t)g^p(t)+\beta(t),\quad
t\in[0,T), \ee where $p>1$ is a constant, $\alpha(t)>0$, $\gamma(t)$ and
$\beta(t)$ are three continuous on $[0,\infty)$ functions. Suppose that 
there
exists a $\mu(t)>0$, $\mu(t)\in C^1[0,\infty)$, such that \be\label{e9}
\alpha(t)\mu^{-p}(t)+\beta(t)\leq
\mu^{-1}(t)[\gamma(t)-\dot{\mu}(t)\mu^{-1}(t)],\qquad t\geq 0, \ee and
\be\label{e10} \mu(0)g(0)<1. \ee Then $T=\infty$, i.e., $g$ exists on
$[0,\infty)$, and 
\be\label{e11} 0\leq g(t)\leq \mu^{-1}(t),\quad t\geq 0.
\ee

\end{lem}
This lemma generalizes a similar result for $p=2$ proved in \cite{R499}.

In Section 2 a method is given for a proof of the following conclusions:
there exists a unique solution $u(t)$ to problem \eqref{e7} for all $t\geq
0$, there exists $u(\infty):=\lim_{t\to \infty}u(t)$, and $F(u(\infty))=f$
: \be\label{e12} \exists ! u(t)\quad \forall t\geq 0; \ \exists
u(\infty);\quad F(u(\infty))=f. \ee

The assumptions on $u_0$ and $a(t)$ under which \eqref{e12} holds for the 
solution to \eqref{e7} are formulated in \thmref{thm1} in Section 2.
Theorem \thmref{thm1} in  Section 2 is our main result. Roughly speaking,
this result says that conclusions \eqref{e12} hold for the solution
to problem \eqref{e7}, provided that $a(t)$ is suitably chosen.

\section{Proofs} Let $|a(t)|:=r(t)>0$. If $a(t)=a_1(t)+a_2(t)$, where
$a_1(t)=\text{Re} a(t),\quad a_2(t)=\text{Im} a(t)$, then \be\label{e13}
|\dot{r}(t)|\leq |\dot{a}(t)|. \ee Indeed, \be\label{e14}
|\dot{r}(t)|=\frac{|a_1\dot{a}_1+a_2\dot{a}_2|}{r(t)}\leq
\frac{r(t)|\dot{a}(t)|}{r(t)}, \ee and \eqref{e14} implies \eqref{e13}.

Let $h\in X^*$ be arbitrary with $\|h\|=1$, and \be\label{e15}
g(t):=(z(t),h);\quad z(t):=u(t)-w_a(t), \ee where $u(t)$ solves \eqref{e7}
and $w_0(t)$ solves \eqref{e4} with $a=a(t)$. By the assumption, $w_a(t)$
exists for every $t\geq 0$. The local existence of $u(t)$, the solution to
\eqref{e7}, is the conclusion of \lemref{lem2}. 
\begin{lem}\label{lem2} If
\eqref{e4}-\eqref{e5} hold, then $u(t)$, the solution to \eqref{e7},
exists locally. \end{lem} \begin{proof} Differentiate equation \eqref{e5}
with respect to $t$. The result is \be\label{e16}
A_{a(t)}(w_a(t))\dot{w}_a(t)=-\dot{a}(t)w_a(t), \ee or \be\label{e17}
\dot{w}_a(t)=-\dot{a}(t)A^{-1}_{a(t)}(w_a(t))w_a(t). \ee Denote
\be\label{e18} \psi(t):=F(u(t))+a(t)u(t)-f. \ee 
For any $\psi\in H$
equation \eqref{e18} is uniquely solvable for $u(t)$ by 
the inverse function theorem, because, by our assumption \eqref{e3},
the Fr\'{e}chet derivative $F'(u(t))+a(t)I$ is boundedly invertible,
and \eqref{e2} implies that the solution $u(t)$ to \eqref{e18} is
continuously differentiable with respect to $t$ if $\psi(t)$ is.
One may solve \eqref{e18} for $u$ and write $u=G(\psi)$, where
the map $G$ is continuously Fr\'{e}chet differentiable because $F$ is.

Differentiate \eqref{e18} and get \be\label{e19}
\dot{\psi}(t)=A_{a(t)}(u(t))\dot{u}(t)+\dot{a}(t)u. \ee If one wants the
solution to \eqref{e18} to be a solution to \eqref{e7}, then one has to
require that \be\label{e20} A_{a(t)}(u(t))\dot{u}=-\psi(t). \ee If
\eqref{e20} holds, then \eqref{e19} can be written as \be\label{e21}
\dot{\psi}(t)=-\psi+\dot{a}(t)G(\psi),\quad G(\psi):=u(t), \ee where
$G(\psi)$ is continuously Fr\'{e}chet differentiable. Thus, equation
\eqref{e21} is equivalent to \eqref{e7} at all $t\geq 0$ if \be\label{e22}
\psi(0)=F(u_0)+a(0)u_0-f. \ee Indeed, if $u$ solves \eqref{e7} then
$\psi$, defined in \eqref{e18}, solves \eqref{e21}-\eqref{e22}.
Conversely, if $\psi$ solves \eqref{e21}-\eqref{e22}, then $u(t)$, defined
as the unique solution to \eqref{e18}, solves \eqref{e7}. Since the
right-hand side of \eqref{e21} is Fr\'{e}chet differentiable, it satisfies
a local Lipschitz condition. Thus, problem \eqref{e21}-\eqref{e22} is
locally, solvable. Therefore, problem \eqref{e7} is locally solvable.\\
\lemref{lem2} is proved. \end{proof} To prove that the solution $u(t)$ to
\eqref{e7} exists globally, it is sufficient to prove the following
estimate \be\label{e23} \sup_{t\geq 0}\|u(t)\|<\infty. \ee
\begin{lem}\label{lem3} 
Estimate \eqref{e23} holds. \end{lem}
\begin{proof} Denote 
\be\label{e24} z(t):=u(t)-w(t), \ee 
where $u(t)$ solves
\eqref{e7} and $w(t)$ solves \eqref{e4} with $a=a(t)$. If one proves that
\be\label{e25} \lim_{t\to \infty}\|z(t)\|=0, \ee then \eqref{e23} follows
from \eqref{e25} and \eqref{e5}: \be\label{e26} \sup_{t\geq 0}\|u(t)\|\leq
\sup_{t\geq 0}\|z(t)\|+\sup_{t\geq 0}\|w(t)\|<\infty. \ee \end{proof} To
prove \eqref{e25} we use \lemref{lem1}. 

Let \be\label{e27} g(t):=\|z(t)\|.
\ee 
Rewrite \eqref{e7} as \be\label{e28}
\dot{z}=-\dot{w}-A_{a(t)}^{-1}(u(t))[F(u(t))-F(w(t))+a(t)z(t)]. \ee Note
that: 
\be\label{e29} \sup_{h\in X^*,\|h\|=1}(\dot{w}(t),h)=\|\dot{w}(t)\|.
\ee

\begin{lem}\label{lem4}
If the norm $\|w(t)\|$ in $X$ is differentiable, then
\be\label{e30}
\|w(t)\|^.\leq \|\dot{w}(t)\|.
\ee
\end{lem}
\begin{proof}
The triangle inequality implies:
\be\label{e31}
\frac{\|w(t+s)\|-\|w(t)\|}{s}\leq \frac{\|w(t+s)-w(t)\|}{s},\quad s>0.
\ee
Passing to the limit $s\searrow 0$, one gets \eqref{e30}.\\
\lemref{lem4} is proved.
\end{proof}

The norm is differentiable if $X$ is strictly convex (see, e.g, \cite{D}).
A Banach space $X$ is called strictly convex if 
$||u+v||<2$  for any $u\neq v\in X$ such that $||u||=||v||=1$. 
A Banach space $X$ is called uniformly convex if for any $\epsilon>0$
there is a $\delta>0$ such that for all $u,v\in B(0,1)$ with
$||u-v||=\epsilon$ one has $||u+v||\leq 2(1-\delta)$. Here $B(0,1)$
is the closed ball in $X$, centered at the origin and of radius one.

Various necessary and sufficient conditions for the Fr\'echet
differentiability of the norm in Banach spaces
are known in the literature (see \cite{DS}), starting with Shmulian's 
paper
of 1940, see \cite{Sh}.

Hilbert spaces, $L^p(D)$ and $\ell^p$-spaces, $p\in (1, \infty)$,   and 
Sobolev spaces $W^{\ell,p}(D)$, $p\in (1, \infty)$, $D\in \R^n$
is a bounded domain, have Fr\'echet 
differentiable norms. These spaces are uniformly convex and they have the 
$E-$property, i.e., if $u_n\rightharpoonup u$ and $||u_n||\to ||u||$
as $n\to \infty$,  then $\lim_{n\to \infty}||u_n-u||=0$.

From \eqref{e17},
\eqref{e29} and \eqref{e3} one gets 
\be\label{e32} \|\dot{w}\|\leq
c_1|\dot{a}(t)|r^{-b}(t)\|w(t)\|,\quad r(t)=|a(t)|, \ee 
where $w(t)=w_a(t)$. Since $\lim_{t\to
\infty}|a(t)|=0$, \eqref{e32} and \eqref{e5} imply \be\label{e33}
\|\dot{w}\|\leq c_2|\dot{a}(t)|r^{-b}(t),\quad c_2=const>0, \ee because
\eqref{e5} implies the following estimate: \be\label{e34} c_1\|w(t)\|\leq
c_2,\quad t\geq 0. \ee 
Inequality \eqref{e13} implies that inequality
\eqref{e33} holds if \be\label{e35} \|\dot{w}\|\leq
c_2|\dot{r}(t)|r^{-b}(t),\quad t\geq 0. \ee Note that \be\label{e36}
F(u)-F(w)=\int_0^1 F'(w+sz)ds z=A(u)z+\int_0^1[A(w+sz)-A(u)]ds z. \ee
Apply $h$ to \eqref{e28}, take $\sup_{h\in X^*,\|h\|=1}$, and use
\lemref{lem4}, relation \eqref{e36}, estimate \eqref{e3}, and
inequality \eqref{e2}, to get:
\be\label{e37} \dot{g}(t)\leq \|\dot{z}(t)\|\leq
c_2|\dot{r}(t)|r^{-b}(t)+c_3r^{-b}(t)g^p-g, \ee where $g(t)$ is defined in
\eqref{e27}, \be\label{e38} p=1+\kappa,\quad c_3:=c_0c_1. \ee Inequality
\eqref{e37} is of the form \eqref{e8} with \be\label{e39}
\gamma(t)=1,\quad \alpha(t)=c_2r^{-b}(t),\quad
\beta(t)=c_2|\dot{r}(t)|r^{-b}(t). \ee 
Choose \be\label{e40}\mu(t)=\lambda
r^{-k}(t),\quad \lambda=const>0,\quad k=const>0 .\ee Then \be\label{e41}
\dot{\mu}\mu^{-1}=-k\dot{r}r^{-1}. \ee 
Let us assume that 
\be\label{e42}
r(t)\searrow 0,\quad \dot{r}<0,\quad |\dot{r}|\searrow 0. \ee Condition
\eqref{e10} implies \be\label{e43} g(0)\frac{\lambda}{r^k(0)}<1, \ee and
inequality \eqref{e9} holds if \be\label{e44}
\frac{c_3r^{-b}(t)r^{kp}}{\lambda^p}+c_2|\dot{r}(t)|r^{-b}(t)\leq
\frac{r^k(t)}{\lambda}(1-k|\dot{r}(t)|r^{-1}(t)),\qquad t\geq 0. \ee 
Inequality
\eqref{e44} can be written as \be\label{e45}
\frac{c_3r^{k(p-1)-b}(t)}{\lambda^{p-1}}+\frac{c_2\lambda|\dot{r}(t)|}
{r^{k+b}(t)}+\frac{k|\dot{r}(t)|}{r(t)}\leq
1. 
\ee

Let us choose $k$ so that
\bee
k(p-1)-b=1, \eee
that is,
\be\label{e46}
k=\frac{b+1}{p-1}.
\ee Choose $\lambda$ as follows:
\be\label{e47}
\lambda=\frac{r^k(0)}{2g(0)}.
\ee 
Then inequality \eqref{e43} holds, and inequality \eqref{e45} can be 
written as:
\be\label{e48}
c_3\frac{r(t)[2g(0)]^{p-1}}{[r^k(0)]^{p-1}}+c_2\frac{r^k(0)}{2g(0)}
\frac{|\dot{r}(t)|}{r^{k+b}(t)}+k\frac{|\dot{r}(t)|}{r(t)}\leq 1,\qquad 
t\geq 0.
\ee
Note that \eqref{e46} implies:
\be\label{e49}
k+b=kp-1.
\ee Choose $r(t)$ so that relations \eqref{e42} hold and 
\be\label{e50}
k\frac{|\dot{r}(t)|}{r(t)}\leq \frac{1}{2},\qquad t\geq 0.
\ee 
Then inequality \eqref{e48} holds if
\be\label{e51}
c_2\frac{[2g(0)]^{p-1}}{r^{b+1}(0)}+c_2\frac{r^k(0)}{2g(0)}
\frac{|\dot{r}(t)|}{r^{kp-1}}\leq \frac{1}{2},\qquad t\geq 0.
\ee
Denote
\be\label{e52}
c_2\frac{r^k(0)}{2g(0)}=c_2\lambda:=c_4.
\ee
Let
\be\label{e53}
c_4\frac{|\dot{r}(t)|}{r^{kp-1}}=\frac{1}{4},\qquad t\geq 0,
\ee 
and $kp>2$. Then equation \eqref{e53} implies
\be\label{e54}
r(t)=\left[t+\frac{4c_4}{kp-2}\frac{1}{r^{kp-2}} \right]^{-\frac{1}{kp-2}}
\left(\frac{kp-2}{4c_4} \right)^{-\frac{1}{kp-2}}. 
\ee 
This $r(t)$ satisfies conditions \eqref{e42}, and equation \eqref{e53} implies:
\be\label{e55}
k\frac{|\dot{r}(t)|}{r(t)}=\frac{kr^{kp-2}(t)}{4c_4},\quad t\geq 0.
\ee 
Recall that $r(t)$ decays monotonically. Therefore, inequality \eqref{e50} 
holds if
\be\label{e56}
\frac{kr^{kp-2}(0)}{4c_4}\leq \frac{1}{2}.
\ee 
Inequality \eqref{e56} holds if
\be\label{e57}
\frac{k}{2}\frac{r^{kp-2-k}(0)}{c_2}2g(0)=\frac{kg(0)}{c_2}r^{k(p-1)-2}(0)
\leq 1.
\ee 
Note that \eqref{e46} implies:
\be\label{e58}
k(p-1)-2=b-1.
\ee
Condition \eqref{e57} holds if $g(0)$ is sufficiently small or $r^{b-1}(0)$ 
is sufficiently large:
\be\label{e59}g(0)\leq \frac{c_2}{k}r^{b-1}(0).\ee 

If $b>1$, then condition \eqref{e59} holds for any fixed $g(0)$ if $r(0)$
is sufficiently large. If $b=1$, then \eqref{e59} holds if $g(0)\leq
\frac{c_2}{k}$. If $b\in (0,1)$ then \eqref{e59} holds either if $g(0)$ is
sufficiently small or $r(0)$ is sufficiently small. 

Consequently, if
\eqref{e54} and \eqref{e59} hold, then \eqref{e53} holds. Therefore,
\eqref{e51} holds if \be\label{e60}
c_3\frac{[2g(0)]^{p-1}}{r^{b+1}(0)}\leq \frac{1}{4}. \ee It follows from
\eqref{e59} that \eqref{e60} holds if \be\label{e61}
c_32^{p-1}\left(\frac{c_2}{k}\right)^{p-1}\frac{1}{r^{p+2b-bp}(0)}\leq
\frac{1}{4}. \ee If $b\in (0,1]$ then \be\label{e62} p-pb+2b>0. \ee Thus,
\eqref{e61} always holds if $r(0)$ is sufficiently large, specifically, if
\be\label{e63} r(0)\geq
[4c_3\left(2c_2 k^{-1}\right)^{p-1}]^{\frac{1}{p(1-b)+2b}}. \ee

We have proved the following theorem.
\begin{thm}\label{thm1} If $r(t)=|a(t)|$ is
defined in \eqref{e54}, and if \eqref{e59} and \eqref{e63} hold, then
\be\label{e64} \|z(t)\|<r^k(t)\lambda^{-1},\quad \lim_{t\to
\infty}\|z(t)\|=0. \ee 
Thus, problem \eqref{e7} has a unique global
solution $u(t)$ and 
\be\label{e65} \lim_{t\to \infty}\|u(t)-y\|=0,
\ee 
where 
\be\label{e66} F(y)=f.
\ee
\end{thm}

\end{document}